\documentclass[12pt,leqno,twoside]{amsart}

\usepackage[utf8]{inputenc}
\usepackage{eucal,mathrsfs}

\usepackage[usenames]{color}

\usepackage[normalem]{ulem}

\usepackage{amsmath,amssymb,amsthm,amsfonts}

\usepackage{float,enumerate,url,mathbbol,mathrsfs,tikz,graphicx,pifont,doi,comment}

\usepackage{mathtools}

\usepackage{tikz, subfigure, xcolor} 
\usepackage{pgfplots}
\usepackage{cite}

\usepackage{dsfont}

\usepackage{hyperref}
\usepackage{comment}
\newtheorem{theorem}{Theorem}[section]
\newtheorem{lemma}[theorem]{Lemma}
\newtheorem{proposition}[theorem]{Proposition}
\newtheorem{corollary}[theorem]{Corollary}

\theoremstyle{definition}
\newtheorem{definition}[theorem]{Definition}
\newtheorem{remark}[theorem]{Remark}
\newtheorem{example}[theorem]{Example}

\numberwithin{equation}{section}

\newcommand{\dist}{\mathrm{dist}}      % distance
\newcommand{\Span}{\mathrm{span}}       % span
\renewcommand{\Im}{{\ensuremath{\mathrm{Im\,}}}} %Imagin????rteil nicht alsfraktur
\renewcommand{\Re}{{\ensuremath{\mathrm{Re\,}}}} %Realteil nicht als fraktur
\providecommand{\norm}[1]{\lVert#1\rVert} %Norm
\providecommand{\abs}[1]{\lvert#1\rvert} % absolut value

\DeclareMathOperator{\Ran}{Ran} 
 
\DeclareMathOperator{\Rank}{Rank}

\DeclareMathOperator{\Ker}{Ker}

\newcommand{\R}{\mathbb{R}}

\newcommand{\C}{\mathbb{C}}

\newcommand{\1}{\mathbb{I}}

\usepackage{eucal}

 \DeclareMathOperator{\Real}{Re}

\newcommand\DM[1]{{\color{red} DM: #1}}

\title[Laplacians with point interactions]{Laplacians with point interactions -- expected and unexpected spectral properties}
\subjclass[2010]{47D06, 34B45}
\keywords{Non-self-adjoint operators, quantum graphs, positive operator semigroups}
\thanks{The second author would like to thank Jochen Glück (Ulm) for helpful discussions.\\
D.M. was partially supported by the Deutsche Forschungsgemeinschaft (Grant 397230547).}

\author{Amru Hussein} 
\address{Amru Hussein, Department of Mathematics,
TU Darmstadt, Schlossgartenstr. 7, 64289 Darmstadt, Germany}
\email{hussein@mathematik.tu-darmstadt.de}

\author{Delio Mugnolo}
\address{Delio Mugnolo, Chair of Analysis, Faculty of Mathematics and Computer Science, FernUniversit\"at Hagen, Germany}
\email{delio.mugnolo@fernuni-hagen.de}

\begin{document}

\begin{abstract}
We study the one-dimensional Laplace operator with point interactions on the real line identified with two copies of the half-line $[0,\infty)$. All possible boundary conditions that define generators of $C_0$-semigroups on  $L^2\big([0,\infty)\big)\oplus L^2\big([0,\infty)\big)$ are characterized. Here, the Cayley transform of the boundary conditions plays an important role, and using an explicit representation of the Green's functions, it allows us to study invariance properties of semigroups.
\end{abstract}

\maketitle

\section{Introduction}\label{sec:intro}
Here, point interactions for the Laplacian on the real line are considered. The real line is realized here as two half-lines $[0,\infty)\dot{\cup} [0,\infty)$ 
%\sout{or $[0,\infty)\times \{1,2\}$}
%\footnote{\DM{Jetzt sehe ich meinen bisherigen notationellen Fehler; auf die zweite Darstellung bestehe ich nicht mehr, wir können sie gerne weglassen.}}
 coupled at the two boundary points. More concretely, we consider realizations $-\Delta(A,B)$ of 
\begin{align*}
-\frac{d^2}{dx} \quad \hbox{in} \quad L^2\big([0,\infty);\C\big)\oplus L^2\big([0,\infty);\C\big)
\end{align*}
with boundary conditions of the form
\begin{align}\label{eq:bc-ab}
A \begin{bmatrix}
\psi_1(0) \\ \psi_2(0)
\end{bmatrix} 
+ B \begin{bmatrix}
\psi'_1(0) \\ \psi'_2(0)
\end{bmatrix} = 0
\quad \hbox{for} \quad A,B\in \C^{2\times 2},
\end{align}
where one has $\psi=(\psi_1,\psi_2)^T\in L^2\big([0,\infty);\C\big)\oplus L^2\big([0,\infty);\C\big)$. Like in~\cite[\S~4]{Mug10}, we regard this setting as a toy model of more complicated quantum graphs.

There are many studies on self-adjoint boundary conditions, cf. \cite{BeKu_book} and references therein, boundary conditions leading to so-called spectral operators, cf. \cite{DSIII}, or boundary conditions related to quadratic forms, cf. \cite{Mug_book}. However, a study of all possible boundary conditions of this form seems to be lacking so far. In this note, we turn to classical Hille--Yosida theory and address the issue of semigroup generation by realizations of the Laplacian with point interactions of the above type.  It turns out that resolvent estimates for $\Delta(A,B)$ are closely related to the behavior of the Cayley transform.  

One could naively expect that imposing two linearly independent boundary conditions is both necessary and sufficient to induce a realization that generates a semigroup, because there are two boundary points and this leads to the rank condition
\begin{align*}
\Rank (A\ B) =2;
\end{align*}
and in fact if $\Rank (A\ B) \neq 2$, then $\sigma(-\Delta(A,B))= \C$, see \cite[Prop. 4.2]{HKS2015}.
However, this rank condition is not yet sufficient to establish basic spectral properties and it turns out that the question of determining when $A,B$ induce a semigroup generator is not  trivial. In a previous work Krej\v{c}i\v{r}\'{i}k, Siegl and the first author, see \cite{HKS2015}, pointed out the importance of the Cayley transform 
\begin{align}\label{eq:Cayley}
\mathfrak{S}(k;A,B):= -(A+ikB)^{-1}(A-ik B), \quad k\in \C,
\end{align}
for basic spectral properties. The condition that $A+ikB$ is invertible for some $k\in\C$ has been used in~\cite{HKS2015} as definition for the notion of \textit{regular boundary conditions}: on general metric graphs irregular boundary conditions can produce very wild spectral features, ranging from empty spectrum -- as in the situation considered here -- to empty resolvent set. 
For the case of one boundary point this cannot occur: the easiest non-trivial case features two boundary points and will be investigated  in detail in the following. 

In the present setting we find out that realizations $\Delta(A,B)$ with irregular boundary conditions have empty resolvent set and thus fail to be generators of $C_0$-semigroups;  surprisingly, it turns out that there are even some regular boundary conditions that do not define generators of $C_0$-semigroups. We will see that not only the mere existence of the Cayley transform is relevant, but also its asymptotic behavior.
The crucial point is that the Cayley transform $\mathfrak{S}(k;A,B)$ appears in a natural way in an explicit formula for the resolvent of $\Delta(A,B)$, which in turn easily allows us to check the conditions of the Hille--Yosida Theorem in its version for analytic semigroups.

Once generation is assessed, we turn to the issue of qualitative properties of the semigroup generated by $\Delta(A,B)$, again in dependence of $A,B$.
It is well-known that relevant features of a semigroup -- in particular, whether it is positive and/or $L^\infty$-contractive -- is tightly related to analogous invariance properties of its generator's resolvent. Using again our machinery, we are then able to formulate sufficient conditions for invariance in terms of properties of $\mathfrak{S}(k;A,B)$. 
In the context of general metric graphs, positivity and Markovian features of semigroups in dependence  of the boundary conditions have been studied already in~\cite[\S~5]{KS2006} -- however only for self-adjoint boundary conditions~\eqref{eq:bc-ab} and giving only sufficient conditions -- and in~\cite[\S~5--6]{CarMug09} for the case of only m-sectorial boundary conditions for which a complete characterisation is obtained, see also \cite{Mug07,Mug10,KanKlaVoi09} for related results. The notion of m-sectorial boundary conditions is explained in Section~\ref{sec:msect} below: roughly speaking, these are boundary conditions that induce realizations of $\Delta(A,B)$ associated with sesquilinear forms.
One step beyond the hitherto discussed invariance properties, we are finally also able to characterize asymptotic positivity of the semigroup -- a rather weak property recently introduced in~\cite{DGK16_b}.

Our note  is organized as follows: 
In Section~\ref{sec:pre} we are going to present our setting, including relevant function spaces and the parametrization of our boundary conditions. Section~\ref{sec:main} contains our main result, Theorem~\ref{thm:generator} as well as a few examples that show its applicability. The proof of Theorem~\ref{thm:generator} is based on a number of technical lemmata, which will be proved in Sections~\ref{sec:msect} and~\ref{sec:caybdd}. Finally, we are going to discuss positivity, asymptotic positivity, and further invariance issues in Section~\ref{sec:invariance}.

%{\color{red} All our findings can be easily extended to star graphs on finitely many edges: we leave the details to the interested reader. \footnote{\AH{Hier waere ich vorsichtig! Wir verwenden sehr stark dass wir nur einen 2-dim raum haben, insbesodnere ist nicht klar wie die beschreibung in Tabelle \ref{fig:cayley} sich auf mehr kanten uebertragen laesst!! Natuerlich ist der 2-Kanten fall enthalten in mehr-kanten faellen, aber eien characterisierung kann man nicht ohne weiteres aus unserem Ergebniss ermitteln.. Meine Meinung ist, diesen Satz wegzulassen.. }} }

%In the next section, we give basic definitions as well as the main results and some illustrative examples. Thereafter, the proofs are discussed in two separate. sections.

%\section{Preliminaries and Main Results}
\section{Function spaces, operators and boundary conditions}\label{sec:pre}
Whenever $I\subset \R$ is an interval, denote by $L^2(I)$  the usual space of complex-valued square integrable function with scalar product $\langle \cdot, \cdot \rangle_{L^2}$. Moreover, let $H^1(I)$ and $H^2(I)$ be the Sobolev spaces of order one and two, and set
  %$H_0^1(I):=\{ \psi\in H^{1}(I)\colon \psi\vert_{\partial I}=0\}$ and  
 $H_0^2(I):=\{ \psi\in H^{2}(I)\colon \psi, \psi'\vert_{\partial I}=0\}$. Then one defines minimal and maximal operators in
\begin{align*}
L^2\big([0,\infty)\big)\oplus L^2\big([0,\infty)\big)
\end{align*} 
by
\begin{align*}
\Delta_{\max} \psi = \psi^{''}, \quad D(\Delta_{\max})&=H^2\big([0,\infty)\big)\oplus H^2\big([0,\infty)\big), \\
\Delta_{\min} \psi = \psi^{''}, \quad D(\Delta_{\min})&=H_0^2\big([0,\infty)\big)\oplus H_0^2\big([0,\infty)\big).
\end{align*}
Since $D(\Delta_{\max})/ D(\Delta_{\min})\cong \C^4$,
any realization
\begin{align*}
\Delta_{\min}  \subset \Delta \subset \Delta_{\max}
\end{align*}
is determined by a subspace $\mathcal{M}\subset \C^4$ and $\Delta=\Delta_{\mathcal{M}}$ with
\begin{align*}
D(\Delta_{\mathcal{M}}) = \left\{\psi \in D(\Delta_{\max})\colon [\underline{\psi}, \underline{\psi}']^T\in \mathcal{M} \right\}, 
\end{align*}
where
\begin{align*}
\underline{\psi} := \begin{bmatrix}
\psi_1(0)\\ \psi_2(0)
\end{bmatrix}
, \quad \underline{\psi}' := \begin{bmatrix}
\psi_1'(0)\\ \psi_2'(0)
\end{bmatrix}, 
\end{align*}
and one sets
\[
 [\underline{\psi}] :=
\begin{bmatrix}
\underline{\psi}\\ \underline{\psi}'
\end{bmatrix}.
\]
For $\dim \mathcal{M}=2$, $\mathcal{M}$ can be represented as kernel of a surjective linear map from $\C^4\rightarrow \C^2$, and hence 
the condition $\dim \mathcal{M}=2$ is equivalent to existence of matrices $A,B\in \C^{2\times 2}$ with $\mathcal{M}=\mathcal{M}(A,B)=\Ker (A\ B)$ and $\Rank (A\ B)=2$. With respect to our goal of studying the generator property of different realizations of Laplacians on $L^2\big([0,\infty)\big)\oplus L^2\big([0,\infty)\big)$, this is the only case which provides enough boundary conditions and we will restrict to it throughout this note. For simplicity, we refer to boundary conditions defined by  $[\underline{\psi}]\in \mathcal{M}(A,B)=\Ker (A\ B)$
in short as \emph{boundary conditions $A,B$}.

 Boundary conditions $A,B$ and $A', B'$ are equivalent if $\mathcal{M}(A,B)=\mathcal{M}(A',B')$, and one sets $$\Delta(A,B):=\Delta_{\mathcal{M}(A,B)}.$$
Note that $A'=CA$ and $B'=CB$ define equivalent boundary conditions whenever $C\in \C^{2\times 2}$ is invertible, since $\Ker (A' \, B')=\Ker (A \, B)$.

The following notion of regularity of boundary conditions has been introduced in \cite[\S~3.2]{HKS2015}. Note that there are also further notions of regularity, in particular the one introduced by Birkhoff, cf. \cite{Birkhoff1908_a, Birkhoff1908_b}, %\footnote{\DM{Tippfehler? Die beiden Referenzen sind identisch.}} 
see also \cite{DSIII}. This regularity assumption does not agree with the one used here, see \cite[\S~3.3]{HKS2015}.

\begin{definition}[Regular and irregular boundary conditions]\label{def:reg}
Let  $A,B$ be boundary conditions with $\Rank (A\ B)=2$. These are called \textit{regular} if $A+ik B$ is invertible for some $k\in \C$, and \textit{irregular} otherwise.
\end{definition}

\begin{remark}\label{rem:irr}
It can be shown that $A,B$  are irregular if and only if $\Rank (A\ B)=2$ and $\Ker A \cap \Ker B\neq \{0\}$, cf. \cite[Prop. 3.3]{HKS2015}.	
\end{remark}

\section{Generator properties and examples}\label{sec:main}

The following is the main result of our paper.
Here $\sigma_{ess}$, $\sigma_{r}$, and $\sigma_{p}$ denote as usual the essential, residual, and point spectrum, respectively.

\begin{theorem}%[Characterization of the infinitesimal generators of $C_0$-semigroups]
\label{thm:generator}
Let the boundary conditions $A,B$ be regular. Then the following assertions hold.
  \begin{enumerate}[(a)]
  	\item $\sigma_{ess}(-\Delta(A,B))=[0,\infty)$, $\sigma_{r}(-\Delta(A,B))=\emptyset$, and $\lambda=k^2\in\sigma_{p}(-\Delta(A,B))$ if and only if $k$ with $\Im k>0$ solves $\det (A+ik B)=0$, and its geometric multiplicity is given by $\dim \Ker (A+ik B)$.
  \item $\Delta(A,B)$ is \emph{not} the infinitesimal generator of a $C_0$-semigroup on $L^2\big([0,\infty)\big)\oplus L^2\big([0,\infty)\big)$ if and only if 
  $\dim \Ker A=0$, $\dim \Ker B=1$, and $P^{\perp}A^{-1}BP^{\perp}=0$, where $P^{\perp}=\mathds{1}-P$ and $P$ denotes the orthogonal projection onto $\Ker B$. 
  \item If $\Delta(A,B)$ is a generator, then the $C_0$-semigroup extends to an analytic semigroup. %in $\Sigma_{\theta, \delta}$ for suitable $\theta\in (0,\pi/2)$ and $\delta\geq 0$, where
  %\begin{align*}
  %\Sigma_{\theta,\delta} = \{ z\in \C\colon |\arg (z-\delta)| < %\theta \}. 
 % \end{align*}
  \item If $\Delta(A,B)$ is a generator and furthermore if any pole $s$ of $k\mapsto \mathfrak{S}(k;A,B)$ or $k\mapsto \mathfrak{S}(-\overline{k};A,B)^*$ with $\Im s>0$ satisfies $\Re s >0$, and $s=0$ is not a pole of any of these functions, 
%\begin{align*}
%s\in \Sigma_\theta= \{z\in \C\setminus \{0\}\colon |\arg (z)|<\pi/4\}
%\end{align*}  
%%  
%  \begin{align*}
%  \Re \{s^2\in \C\colon  \Im s \geq 0, s \hbox{ is a pole of } k\mapsto \mathfrak{S}(k;A,B) \}>0, \\
%\Re \{s^2\in \C\colon  \Im s \geq 0, s \hbox{ is a pole of } k\mapsto \mathfrak{S}(-\overline{k};A,B)^* \}>0,
%  \end{align*}
  then the semigroup generated by $\Delta(A,B)$ is uniformly bounded. 
%  \footnote{\DM{Wenn ich richtig den Beweis verstehe, Pole der Cayleytransform mit Imaginärteil $\le 0$ sind fürs Spektrum irrelevant. Wenn es so ist, sollten die obigen Bedingungen nicht eher
%  \begin{align*}
%  \Re \{s^2\in \C\colon  s \hbox{ is a pole of } k\mapsto \mathfrak{S}(k;A,B) \hbox{ and }\Im s>0\}>0, \\
%\Re \{s^2\in \C\colon  s \hbox{ is a pole of } k\mapsto \mathfrak{S}(-\overline{k};A,B)^* \hbox{ and }\Im s>0\}>0,
%  \end{align*}
%sein?  Oder -- vielleicht einfacher --: all poles $s$ of both $k\mapsto \mathfrak{S}(k;A,B) $ and $k\mapsto \mathfrak{S}(-\overline{k};A,B)^*$ with $\Im s>0$ satisfy $\Re s>\Im s$.}\AH{Habe s nochmal geaendert}}
 \item If $A=L+P$ and $B=P^{\perp}$ for an orthogonal projection $P$ in $\C^2$, $P^{\perp}=\mathds{1}-P$, and $L\in \C^{2\times 2}$ with $P^{\perp}LP^{\perp}=L$, then $\Delta(A,B)$ is even the generator of a cosine operator function and hence of an analytic semigroup of angle $\frac{\pi}{2}$ on $L^2\big([0,\infty)\big)\oplus L^2\big([0,\infty)\big)$. This semigroup is always quasi-contractive, and in fact contractive if the numerical range of $L$ is contained in $\{z:\Re z\le 0\}$.
  \end{enumerate}
\end{theorem}

If $\Delta(A,B)$ is a generator, then the semigroup consists of operators that are bounded on $L^2$ and map $L^2$ into $H^2\hookrightarrow L^\infty$, hence are integral operators by the Kantorovich--Vulikh Theorem.

\begin{remark}[Irregular boundary conditions do not define generators]
If $A,B$ are irregular, $\Delta(A,B)$ cannot be a generator of a $C_0$-semigroup since $\sigma(\Delta(A,B))=\C$. For the case of general finite metric graphs, determining spectra and resolvent estimates for irregular boundary conditions is more involved.%\footnote{\DM{Es ist nicht ganz klar, was ``the question'' ist. Was ist überhaupt mit diesem Satz gemeint? \textit{Irregular} Randbedingungen induzieren niemals einen Generator, oder?}\AH{Jetzt besser?}}
\end{remark}

	\begin{remark}[Multiplicity of eigenvalues]\label{rem:mult}
		For regular boundary conditions $A,B$, the geometric multiplicity of an eigenvalue $-k^2$ of $\Delta(A,B)$ is at most two, and equal to two if and only if  $A+ik B=0$. 
		This implies that $\Ker B=\Ker A=\{0\}$, and 
		that there are equivalent boundary conditions $A' = l \cdot\mathds{1}$ and $B'=\mathds{1}$ with $\Re l{{=\Im k>}}0$.
	\end{remark}

Unfortunately, we are not able to determine the semigroup's analyticity angle in the general case; as a matter of fact, we cannot exclude that $\Delta(A,B)$ is always the generator of a cosine operator function. Indeed, the proof of (c) shows that the spectrum of $\Delta(A,B)$ is always contained in a parabola centered around the real axis; this is 	a necessary condition for generation of a cosine operator function, cf.~\cite[Thm.~3.14.18]{AreBatHie01}.

\begin{proof}
The proof of (a) can be deduced from \cite[Sec.~4]{HKS2015}: The statement on the residual spectrum follows from  \cite[Prop. 4.6]{HKS2015} for the case of only external edges. Essential spectra are discussed in  \cite[Prop. 4.11]{HKS2015}. Note that for non-self-adjoint
operators there are various notions of the essential spectrum. Five types, defined in
terms of Fredholm properties and denoted by $\sigma_{e_j}$ for $j = 1, \ldots, 5$, are discussed
in detail in \cite[Chap. IX]{EE_book}. All these essential spectra coincide for  self-adjoint $T$, but
for closed non-self-adjoint $T$ in general one has  only the inclusions $\sigma_{e_j}(T)\subset \sigma_{e_j}(T)$ 
for $j <i$. However, here one even has $\sigma_{ess}(-\Delta(A,B))=\sigma_{e_j}(-\Delta(A,B))$ for $i=1,\ldots,5$. The statement on the eigenvalues follows from the \textit{Ansatz} for the eigenfunctions
\begin{align*}
\psi(x; k)=\begin{bmatrix}
\alpha_1(k)e^{ikx_1} \\ \alpha_2(k)e^{ikx_2} 
\end{bmatrix},
\end{align*} 
which is square integrable only if $\Im k>0$, and there are non-trivial $\alpha_j(k)$, $j=1,2$, such that $\psi(\cdot;k)\in D(\Delta(A,B))$ if and only if $\det (A+ik B)=0$, and the geometric multiplicity is given by $\dim \Ker (A+ik B)$.
For part (b), uniform boundedness of the Cayley transform is characterized in Lemma~\ref{lem:msec}, Lemma~\ref{lem:Ok} and Lemma~\ref{lem:11} below. The corresponding resolvent estimates are given in Lemma~\ref{lem:resbdd} and~\ref{lem:resbad} below, where Lemma ~\ref{lem:resbad} discusses the case of non-generators. Lemma~\ref{lem:resbdd} implies that for $\omega>|S|$, where $S$ is a set of singularities of $\mathfrak{S}(k;A,B)$ defined there,  $\Delta(A,B)$ is a closed densely defined operator with $[\omega^2,\infty)\in \rho(\Delta(A,B))$. %\footnote{\DM{Vermutlich ist hier $[\omega^2,\infty)\subset \rho(\Delta(A,B))$ gemeint.}}
For any sector 
\begin{align}
\Sigma_{\theta}:=\{k \in \C \colon  \Im k>0, |\Re k| \leq \tan(\theta)\Im k\}, \quad \theta\in (0,\pi/2),
\end{align}
one has $|k| \cong |\Re k| + \Im k\leq (1+\tan(\theta)) \Im k$, and therefore
 \begin{align*}
 \norm{(-\Delta(A,B) -k^2 )^{-1}} \leq \frac{1}{|k|^2}+\frac{C_{\omega}}{|k|^2} \leq \frac{1+C_{\omega}}{|k|^2}, \quad k\in \Sigma_c, \quad |k|>\omega.
 \end{align*}	
In particular, $\Delta(A,B)$ is sectorial on the sector $\Sigma_{2\theta}-\omega^2\sin(\pi-2\theta)$. 
Shifting the sector allows avoiding the two poles of $\mathfrak{S}(k;A,B)$: this finishes the proof of (b) and (c), whereas (e) is proved in Lemma~\ref{lem:cosine}.

A necessary condition for boundedness of a semigroup is that the spectrum of its generator $\mathcal A$ is contained in $\{z\in \C:\Re z\le 0\}$. 
To prove (d), recall that by a celebrated result due to Gomilko~\cite{Gom99}, for semigroups acting on a Hilbert space $H$ boundedness is equivalent to said spectral inclusion \textit{and} the additional condition
$$\sup_{\delta>0} \delta \int_{\delta-i\infty}^{\delta+i\infty} \left(\norm{(\mathcal A-\lambda)^{-1} f}^2 +\norm{(\mathcal A^*-\lambda)^{-1} f}^2\right) |d\lambda|<\infty\qquad\hbox{for all } f\in H.$$
Here, the kernel of $(-\Delta(A,B)-k^2)^{-1}$ is given below by \eqref{eq:res_infty} and the kernel of $(-\Delta(A,B)^*-\overline{k^2})^{-1}$ is given by the adjoint kernel $r_{A,B}(y,x;-\overline{k})^*$. Analogously to Lemma~\ref{lem:resbdd} one can estimate the resolvent norm away from the singularities of $\mathfrak{S}(k;A,B)$ and $\mathfrak{S}(-\overline{k};A,B)^*$. These singularities are finitely many and have by assumption a finite, strictly positive distance to the imaginary axis. In particular, the estimate in Lemma~\ref{lem:resbdd} implies sectoriality in sectors with vertex zero
\begin{align*}
\norm{(\Delta(A,B)-\lambda)^{-1}} \leq \frac{C}{|\lambda|} \quad \hbox{for}\quad \lambda \in \Sigma_{\theta}:= \{z\in \C\setminus \{0\} \colon |\arg z|> \theta\}, \quad \theta \in (0,\pi/2),
\end{align*}  
and an analogous estimate holds for $\norm{(\Delta(A,B)^*-\lambda)^{-1}}$. Therefore,
\begin{multline*}
\sup_{\delta>0} \delta \int_{\delta-i\infty}^{\delta+i\infty} \left(\norm{(\Delta(A,B)-\lambda)^{-1} f}^2 +\norm{(\Delta(A,B)^*-\lambda)^{-1} f}^2\right) |d \lambda| \\
\leq C \sup_{\delta>0} \delta  \int_{\delta-i\infty}^{\delta+i\infty} \frac{1}{|\delta +i\lambda|^{2}} |d\lambda| 
\leq C \sup_{\delta>0} \delta \int_{|\lambda|>\delta} \frac{1}{|\lambda|^{2}} d|\lambda| = 2C <\infty.
\end{multline*}
This completes the proof.
\end{proof}

\begin{comment}
{\color{purple}
\begin{remark}[Uniform boundedness of the semigroup]
%Recall that for analytic semigroups the growth bound is equal to the spectral bound, see e.g. \cite[Cor. IV.3.12]{EngelNagel}. Hence to determine if the semigroup generated by $\Delta(A,B)$ is uniformly bounded it is sufficient  by Theorem~\ref{thm:generator} (a) to compute the eigenvalues of $\Delta(A,B)$, and the semigroup is uniformly bounded if and only if $\Re \sigma_p(\Delta(A,B))\leq 0$.

%Note that for general $C_0$-semigroups the question of uniform boundedness is more involved, and this property has been characterized in celebrated result by Gomilko~\cite{Gom99}. 
%%
% 
%Since all semigroups appearing here are analytic, more involved conditiosfor the semigroup  
%
%
\end{remark}
}
\end{comment}

The generator property is traced back to the uniform boundedness of the  Cayley transform $k\mapsto \mathfrak{S}(k;A,B)$ outside a compact set containing its poles, where for irregular boundary conditions one might set $\mathfrak{S}(k;A,B)=\infty$. Some  cases for the possible behavior of the Cayley transform are illustrated in the following examples.

\begin{example}[Boundary conditions defining operators associated with sectorial forms]\label{ex:sec}
Let 	
\begin{align*}
A=\begin{bmatrix}
A_{11} & 0 \\ A_{21} & 0
\end{bmatrix}\quad \hbox{and} \quad B=\begin{bmatrix}
1 & 0 \\ 0 & 1
\end{bmatrix}.
\end{align*}
for any $A_{11},A_{22}\in \C$: the boundary conditions $A,B$ correspond to $A\underline{\psi}+B\underline{\psi'}=0$, i.e.,
\[
A_{11}\psi_1(0)+\psi'_1(0)=0, \quad A_{21}\psi_1(0)+\psi_2'(0)=0.
\]
Then $A,B$ are regular since $\det(A+ik B)=ik(A_{11}+ik)\neq 0$ for $k\notin\{0, iA_{11} \}$.%\footnote{\DM{??? Es sollte ja `` for $k\notin\{0, iA_{11} \}$'' heißen. Dementsprechend ist $-A_{11}^2$ genau dann ein Eigenwert, wenn $\Re A_{11}>0$. Ich möchte auch erwähnen, dass die Eigenfunktionen in diesem Fall
%\[
%\psi_1(x)=A_{11}e^{-A_{11}x},\qquad \psi_2(x)=A_{21}e^{-A_{11}x}
%\]
%sind.}} 
The Cayley transform  
	\begin{align}\label{eq:cayley_msec-exa}
	\mathfrak{S}(k;A,B) 
	=\begin{bmatrix}
	-(A_{11} + ik)^{-1} (A_{11} - ik)& 0  \\
	(ik)^{-1}A_{12} [(A_{11} + ik)^{-1} (A_{11} - ik)  -1] & \mathds{1} 
	\end{bmatrix}
	\end{align}  
	 is uniformly bounded  away from its singularity $\{0, iA_{11} \}$, where $0$ is in fact a removable singularity.
	Since $\dim \Ker B=0$, by Theorem~\ref{thm:generator} $\Delta(A,B)$ generates an analytic semigroup;  if $\Im  i A_{11}>0$, then $A_{11}^2$ is a (simple, by Remark~\ref{rem:mult})  eigenvalue of $\Delta(A,B)$, %where the eigenfunction is
%	\begin{align*}
%\psi = \begin{bmatrix}
%A_{11}e^{iA_{11}x_1} \\
%A_{21}e^{iA_{11}x_2}
%\end{bmatrix},%
%	\end{align*}
	and $\sigma_{ess}(-\Delta(A,B))=[0,\infty)$.
	 %in particular, the semigroup generated by $\Delta(A,B)$ is unbounded
	Note that $-\Delta(A,B)$ is associated with the sesquilinear form defined by
	\begin{equation}\label{eq:form-0}
	\delta_{A,B}[\psi]= \norm{\psi'}^2_{L^2} - \langle A\underline{\psi}, \underline{\psi}\rangle_{\C^2},  \qquad \psi \in  H^1([0,\infty))\oplus H^1([0,\infty))
	\end{equation}
	and hence sectorial, in particular, the semigroup generated by $\Delta(A,B)$ is contractive if the numerical range of $A$ is contained in the left halfplane: this is the case if and only if $A_{21}=0$.  %With an abuse of notation, 
	We will refer to boundary conditions of this type as \emph{m-sectorial}.
\end{example}

In a more general setting the question if $-\Delta(A,B)$ is associated with a form of the type given in \eqref{eq:form-0} is discussed in \cite{Hussein2014}. 
The following is a prominent example from the theory of $\mathcal{PT}$-symmetric operators, and it is discussed for instance in \cite[Example 3.5]{HKS2015} and also in the references given there.

%A particular of boundary conditions is a prominent example from the theory of $\mathcal{PT}$-symmetric operators, and it is discussed for instance in \cite[Example 3.5]{HKS2015} and also in the references given there.

\begin{example}[Boundary conditions defining operators not associated with sectorial forms]
\label{ex:new}
Consider
\begin{align*}
A_{\tau}=\begin{bmatrix}
1 & -e^{i\tau} \\ 0 & 0
\end{bmatrix}\quad \hbox{and} \quad B_\tau=\begin{bmatrix}
0 & 0 \\ 1 & e^{-i\tau}
\end{bmatrix}, \quad \tau \in [0,\pi/2),
\end{align*}
leading to the boundary conditions
\[
\psi_1(0)=e^{i\tau}\psi_2(0),\quad \psi_1'(0)=-e^{-i\tau}\psi'_2(0).
\]
Here, $\det (A_\tau+ikB_\tau)=2ki \cos \tau\ne 0$ and hence by Theorem~\ref{thm:generator} $\Delta(A_\tau,B_\tau)$ has no eigenvalues. 
Integration by parts gives
\begin{equation*}
\langle -\Delta(A_{\tau},B_{\tau})\psi, \psi\rangle_{L^2} =
\norm{\psi'}_{L^2}^2 + (1-e^{2i\tau}) \psi_2(0)\overline{\psi_2}'(0), \quad \psi \in D(\Delta(A_\tau,B_\tau)). 
\end{equation*}
The trace of the derivative cannot be balanced by $\norm{\psi'}_{L^2}^2$, hence in particular $\psi\mapsto \langle -\Delta(A_{\tau},B_{\tau})\psi, \psi\rangle_{L^2}$ does not define a closed sesquilinear	 form: indeed, the numerical range of this form is the entire complex plane. Nevertheless, $\Delta(A_{\tau},B_{\tau})$ does generate an analytic semigroup, as in fact $\Delta(A_{\tau},B_{\tau})$ is similar to the one-dimensional Laplacian on $\R$. Observe that because $\Delta(A_{\tau},B_{\tau})$ is not dissipative, the semigroup it generates cannot be contractive; it is bounded, though, due to its similarity with the Gaussian semigroup on $\R$. Observe that $A_\tau,B_\tau$ are irregular boundary conditions for $\tau=\frac{\pi}{2}$. 
\end{example}
%\footnote{\DM{Sehr interessant! Danke, ich kannte dieses Beispiel nicht mehr. Der Witz ist vermutlich, dass $\Delta(A,B)$ dem Laplace auf $\R$ zwar ähnlich, aber nicht unitär äquivalent ist, oder?}}

The following two examples are slight modifications of cases discussed in~\cite[Section XIX.6, page 2373]{DSIII}.

\begin{example}[Intermediate boundary conditions]
\label{ex:intermediate}
	Consider
	\begin{align*}
A=\begin{bmatrix}
1 & 0 \\ 0 & 1
\end{bmatrix}\quad \hbox{and} \quad B=\begin{bmatrix}
0 & 0 \\ -1 & 0
\end{bmatrix},
	\end{align*}
and the boundary conditions $A\underline{\psi}+B\underline{\psi'}=0$, i.e.,
\[
\psi_1(0)=0, \quad \psi_1'(0)=\psi_2(0).
\]
	Then $\det(A+ikB)=1$  for all $k\in\mathbb C$, i.e., $A,B$ are regular. Furthermore,  $\dim \Ker A=0$ and $\dim \Ker B=1$, but $P^{\perp}BP^{\perp}=0$ and $PBP^{\perp}=-1\neq 0$ and
	\begin{align*}
	\mathfrak{S}(k;A,B) = -\begin{bmatrix}
	1 & 0 \\ 2ik & 1
	\end{bmatrix}, \quad k\in \C.
	\end{align*}
We conclude from Theorem~\ref{thm:generator} that $\Delta(A,B)$ does not generate an analytic semigroup on $L^2\big([0,\infty)\big)\oplus L^2\big([0,\infty)\big)$, although its (purely essential) spectrum is $[0,\infty)$.
%For $\Delta(A,B)$ in $L^2([0,1])$, the spectrum consists of eigenvalues $k^2$ of geometrical multiplicity one where $k$ solves $\sin(k)=k$. Asymptotically, the eigenvalues are located at the points $an^2 + ib \ln n +\ldots$, where the ration of $a$ and $b$ is real. Here, we found out that    
%\begin{align*}
%\Vert(\Delta(A,B)+k^2)^{-1}\Vert \geq c k \hbox{ for } k \to \infty, %\quad c>0,
%\end{align*}   	
%and hence although the spectrum lies in a sector this cannot be a generator of a $C_0$-semigroup.
\end{example}

\begin{example}[Totally degenerate boundary conditions]\label{ex:degenerate}
	Consider 
	\begin{align*}
	A=\begin{bmatrix}
	1 & 0 \\ 0 & 0
	\end{bmatrix}\quad \hbox{and} \quad B=\begin{bmatrix}
	0 & 0 \\ 1 & 0
	\end{bmatrix}.
	\end{align*}
	Then $\Rank (A\ B)=2$, but $\det (A + ik B)=0$ for any $k\in \C$, and hence $A,B$ are irregular. %The spectrum is empty and for the resolvent one can show that
%	\begin{align*}
%	\Vert(\Delta(A,B)+k^2)^{-1}\Vert \geq c, \quad  k \to \infty, \quad c>0,
%	\end{align*}   	
%	and this cannot be a generator of a $C_0$-semigroup.
\end{example}

\begin{comment}
\begin{theorem}[Criterium for eventual positivity]\label{thm:pos}
Let $A=L+P$ and $B=P^\perp$ be $m$-sectorial boundary conditions, and assume that $A,B$ are real. Consider $\Delta(A;B)$ in $L^2([0,a])$. If 
\begin{align*}
\min_{\underline{\psi}\in \Ran P^\perp, \norm{\underline{\psi}}=1} \Real \langle \underline{\psi}, L+M(a) \underline{\psi}\rangle_{\C^2}=0
\end{align*}
and the projection onto $\Ker AX_0+BY_0$ is positive. Then $\Delta(A;B)$ is eventually positive. 
\end{theorem}
This gives a criterion which is can be verified in terms of linear algebra. It generalizes in a straight forward way to general compact finite metric graphs.
\end{comment}

\section{Cayley transforms}\label{sec:msect}

In this section we are going to derive properties of the Cayley transform that are essential in the proof of Theorem~\ref{thm:generator}.

\subsection{M-sectorial boundary conditions}
For regular boundary conditions the Cayley transform \eqref{eq:Cayley} is well-defined except for  at most two $k\in \C$. 
One important class of boundary conditions are related to quadratic forms. 

\begin{definition}
Boundary conditions $A,B$ are said to be \emph{m-sectorial} if there exist $L,P\in \C^{2\times 2}$ such that $P$ is an orthogonal projection, $P^\perp=\mathds{1}-P$, and $L=P^\perp L P^\perp$, and such that  $A=L+P$ and $B=P^\perp$.
\end{definition}

The reason for this name is that whenever $A,B$ are m-sectorial boundary conditions,  $-\Delta (A,B)$ is associated with the sectorial sesquilinear form, cf. e.g. \cite[Def.~1.7]{Ouh05} for this notion,
\begin{multline}\label{eq:form}
\delta_{P,L}[\psi]= \norm{\psi'}^2_{L^2} - \langle LP^\perp\underline{\psi}, P^\perp\underline{\psi}\rangle_{\C^2}, \, \psi \in \{\psi \in H^1([0,\infty))\oplus H^1([0,\infty))\colon P\underline{\psi}=0\};
\end{multline}
$-\Delta (A,B)$ is hence an m-sectorial operator and $\Delta (A,B)$ generates an analytic semigroup. 
M-sectorial boundary conditions are in particular regular
since 
\begin{align*}
A+ik B = \begin{bmatrix}
L +ik P^{\perp} & 0 \\ 0 & P
\end{bmatrix}
\end{align*} 
is invertible for $k>\norm{L}$.
The Cayley transform can be estimated as follows.

We first consider the case $\dim \Ker B=1$: then $\dim \Ran L\leq 1$, and with respect to $\Ran P$ and $\Ran P^{\perp}$
one obtains a block decomposition 
\begin{align*}
A\pm ik B=\begin{bmatrix}
L \pm ik \mathds{1}& 0 \\
0 & \pm ik \mathds{1} & 
\end{bmatrix}\quad \hbox{and}\quad 
\mathfrak{S}(k;A,B) =
\begin{bmatrix}
-(L + ik \mathds{1})^{-1}(L - ik \mathds{1})& 0 \\
0 & \mathds{1} & 
\end{bmatrix}.
\end{align*} 
For $A$ invertible and $B=\mathds{1}$ one has
\begin{align*}
\mathfrak{S}(k;A,B) = -(A+ik\mathds{1})^{-1} (A-ik\mathds{1}) = -(A/ik+\mathds{1})^{-1} (A/ik-\mathds{1}),
\end{align*}
and hence 
\begin{align*}
\norm{\mathfrak{S}(k;A,B)} \leq \frac{2}{1-\tfrac{\norm{A}}{|k|}} \quad \hbox{for } |k|> \Vert A\Vert.
\end{align*}
Therefore, $\mathfrak{S}(k;A,B)$ is uniformly bounded away from its poles, i.e., outside a compact set. 

If $B=\mathds{1}$  and $\dim \Ran L= 1$, one obtains a block decomposition with respect to $\Ker L$ and  $(\Ker L)^{\perp}$
\begin{align*}
A\pm ik B=\begin{bmatrix}
L_{11} \pm ik \mathds{1}& 0 \\
L_{12} & \pm ik \mathds{1} 
\end{bmatrix},
\end{align*} 
and hence
\begin{equation}\label{eq:cayley_msec}
\begin{split}
	\mathfrak{S}(k;A,B) &=-\begin{bmatrix}
(L_{11} + ik \mathds{1})^{-1}& 0 \\
-(ik)^{-1} L_{12}(L_{11} + ik \mathds{1})^{-1} & + (ik)^{-1} \mathds{1} 
\end{bmatrix}
\begin{bmatrix}
L_{11} - ik \mathds{1}& 0  \\
L_{12} & - ik \mathds{1} 
\end{bmatrix} \\
&=\begin{bmatrix}
-(L_{11} + ik \mathds{1})^{-1} (L_{11} - ik \mathds{1})& 0 \\
(ik)^{-1}L_{12} [(L_{11} + ik \mathds{1})^{-1} (L_{11} - ik \mathds{1})  -\mathds{1}] & \mathds{1} 
\end{bmatrix}.
\end{split}  
\end{equation}
Similarly to the case of $A$ invertible and $B=\mathds{1}$,  using \eqref{eq:cayley_msec} one can show that  $\mathfrak{S}(k;A,B)$ is uniformly bounded away from its poles for general m-sectorial boundary conditions. This is summarized in the following. 
\begin{lemma}\label{lem:msec}
	Let $A,B$ define $m$-sectorial boundary conditions. Then $\mathfrak{S}(k;A,B)$ is uniformly bounded outside a compact set.
\end{lemma}

Depending now on the dimension of $\Ker A$ and $\Ker B$ one can distinguish the following cases listed in Table~\ref{fig:tab}, where the cases $\dim \Ker A= 1$, $\dim \Ker B= 2$, and $\dim \Ker A= 2$, $\dim \Ker B= 1$ collide with the rank condition, and hence are excluded. We have already remarked that for $\Rank(A\ B)\neq 2$ one has $\sigma(\Delta(A,B))=\C$.

\begin{table}[H]
	\begin{tabular}{c c c c}
		$\dim \Ker A$ & $\dim \Ker B$ & equiv. b.c. & $-\Delta(A,B)$ \\ \hline
		& & & \\
		0 & 0 & $A'=B^{-1}A$ and $B'=\mathds{1}$ & m-sectorial \\
		1 & 0 & $A'=B^{-1}A$ and $B'=\mathds{1}$ &  m-sectorial\\
		2 & 0 & $A'=B^{-1}A$ and $B'=\mathds{1}$ & m-sectorial \\
		0 & 2 & $A=\mathds{1}$ and $B=0$ &  m-sectorial\\
		0 & 1 & some block representation & regular  \\
		1 & 1 & some block representation & regular or irregular \\[1pt]
	\end{tabular}
	\caption{Different cases of boundary conditions}\label{fig:tab}
\end{table}

\subsection{The case $\dim \Ker A =0$ and $\dim \Ker B =1$}

\begin{lemma}%[Rank condition and regularity]
	Let $\dim \Ker A =0$ and $\dim \Ker B =1$. Then $\Rank (A\; B)=2$, equivalent boundary conditions are given by
	\begin{align*}
	A'=\mathds{1}\quad \hbox{and} \quad B'=A^{-1}B,
	\end{align*} 
	and the boundary conditions $A,B$ are regular. 
\end{lemma}
\begin{proof}
	First, since $A$ is invertible, its columns are linearly independent and therefore $\Rank (A\;B)\geq 2$, and  $A'$, $B'$ define equivalent boundary conditions. Furthermore, since $\Ker A \cap \Ker B= \{ 0\}$ these boundary conditions are regular.
\end{proof}
% Let $\dim \Ker A =0$ and $\dim \Ker B =1$, then  Since $A$ invertible, one can consider $A'=\mathds{1}$ and $B'=A^{-1}B$, and letting 
Let $P$ be the orthogonal projection onto $\Ker B$ and $P^{\perp}=\mathds{1}-P$, then without loss of generality, consider   
\begin{align*}
A= \mathds{1}\quad \hbox{and} \quad 
B= \begin{bmatrix} P^{\perp}BP^{\perp} & 0 \\
P BP^{\perp} & 0
\end{bmatrix}.
\end{align*}

\begin{lemma}%[Cayley transform]
\label{lem:Ok}
	Let $\dim \Ker A =0$ and $\dim \Ker B =1$. Then the Cayley transform $\mathfrak{S}(\cdot;A,B)$ is uniformly bounded outside a compact set containing its only possible pole if and only if $P^{\perp}BP^{\perp}\neq 0$. If $P^{\perp}BP^{\perp}= 0$, and hence $P BP^{\perp}\neq 0$, then     
	$\norm{\mathfrak{S}(k;A,B)}=\mathcal{O}(|k|)$ for $|k|\to \infty$
\end{lemma}
\begin{proof}
	Here, since $\dim \Ker B=1$, $\Ran P^{\perp}=\Span\{p_1 \}$ and $\Ran P=\Span\{p_2 \}$, where $\{p_1,p_2\}$ is an orthonormal basis of $\C^2$. In this basis   
	\begin{align*}
	(A\pm ikB)&= \begin{bmatrix}1 \pm ik B_{11} & 0 \\
	\pm ik B_{21} & 1
	\end{bmatrix}, \quad
	(A\pm ikB)^{-1} =  
	\frac{1}{1 \pm ik B_{11} }  \begin{bmatrix} 1 & 0 \\
	\mp ik B_{21} & 1 \pm ik 1
	\end{bmatrix}.
	\end{align*}
	The Cayley transform is then
	\begin{align*}
	\mathfrak{S}(k;A,B)&= -(A+ikB)^{-1}(A-ikB)
	% &=-\frac{1}{1 + ik PBP }  \begin{bmatrix} 1 & 0 \\
	% -ik P^{\perp} BP & 1 + ik PBP
	% \end{bmatrix}
	% \begin{bmatrix}1 - ik PBP & 0 \\
	% -ik P^{\perp} BP & 1
	% \end{bmatrix}\\
	= -\frac{1}{1 + ik B_{11} } \begin{bmatrix}1 - ik B_{11}  & 0 \\
	-2ik B_{21} & 1
	\end{bmatrix}.
	\end{align*} 
	For $B_{11}\neq 0$ this is uniformly bounded away from the pole $k = i/B_{11}$. For $B_{11}= 0$ there are no poles, and $\dim \Ker B=1 $ implies that $B_{21}\neq 0$. In this case  $\norm{\mathfrak{S}(k;A,B)}=\mathcal{O}(|k|)$ for $|k|\to \infty$. 
\end{proof}

%  Provided $PBP\neq 0$ $ \Sigma(A,B,k)$ is uniformly bounded  for $k$ large.
%  For $PBP= 0$
%  \begin{align*}
%  \Sigma(A,B,k)&= -(A+ikB)^{-1}(A-ikB)
%  = -\begin{bmatrix}1  & 0 \\
%  -2ik P^\perp BP & 1
%  \end{bmatrix}
%  \end{align*} 
%  this is no longer true, and equivalence of norms in $\C^{2\times 2}$ gives that $\norm{\Sigma(A,B,k)} = \mathcal(O)(|k|)$ for $|k|\to \infty$.
%  

\subsection{The case $\dim \Ker A =1$ and $\dim \Ker B =1$.}
%Under our standing assumption $\Rank (A\; B)=2$, 

In this subsection we focus on the case of $\dim \Ker A =\dim \Ker B =1$. Denote by $Q^{\perp}$ the orthogonal projection on $\Ran A$, $Q=\mathds{1}-Q^{\perp}$, and as before $P$ the  orthogonal projection on $\Ker B$, where each $Q,Q^{\perp}$ and $P,P^{\perp}$  has one-dimensional range.
Then
\begin{align}\label{eq:repQP}
A= \begin{bmatrix} Q^{\perp}AP^{\perp} & Q^{\perp}AP \\
0 & 0
\end{bmatrix}\quad \hbox{and} \quad 
B= \begin{bmatrix} Q^{\perp}BP^{\perp} & 0 \\
Q BP^{\perp} & 0
\end{bmatrix}.
\end{align}

\begin{lemma}%[Rank condition and regularity]
\label{lem:11rr}
	Let $\dim \Ker A =\dim \Ker B =1$. Then  $\Rank (A\; B)=2$ if and only if $Q BP^{\perp}\neq 0$.
	The boundary conditions $A,B$ are irregular if and only if $\Ker A=\Ker B$, i.e. if $Q^{\perp}AP=0$. 
\end{lemma}
\begin{proof}
From \eqref{eq:repQP} one deduces $(A\,B)$ is surjective if and only if $Q BP^{\perp}\neq 0$. Recall that $A,B$ are irregular if and only if $\Ker A = \Ker B$, see \cite[Prop. 3.3]{HKS2015}, and here \eqref{eq:repQP} implies that $\Ker A = \Ker B$ if and only if  $Q^{\perp}AP=0$. 
\end{proof}

\begin{lemma}%[Cayley transform]
\label{lem:11}
	Let $A,B$ define regular boundary conditions with $\dim \Ker A =\dim \Ker B =1$. Then the Cayley transform $\mathfrak{S}(\cdot;A,B)$ has one possible pole, and away from this    
$\mathfrak{S}(\cdot;A,B)$ is uniformly bounded. 
\end{lemma}
\begin{proof}
	
	Note that $\Ran P^{\perp}=\Span\{p_1 \}$ and $\Ran P=\Span\{p_2 \}$, where $\{p_1,p_2\}$ is an orthonormal basis of $\C^2$. For  $A,B$, regular, in this basis, 
	equivalent boundary conditions are
	\begin{align*}
	A= \begin{bmatrix} A_{11} & 1 \\
	0 & 0
	\end{bmatrix}\quad \hbox{and} \quad 
	B= \begin{bmatrix} B_{11} & 0 \\
	1 & 0
	\end{bmatrix}
	\end{align*}
	since

	Let $\Ran P=\Span\{p_1 \}$, $\Ran P^\perp=\Span\{p_2 \}$, and $\Ran Q=\Span\{q_1 \}$, $\Ran Q^\perp=\Span\{q_2 \}$ where $\{p_1,p_2\}$ and $\{q_1,q_2\}$  are orthonormal basis of $\C^2$. Now, a coordinate change from $\{q_1,q_2\}$ to $\{p_1,p_2\}$ is given by a unitary $U$, and hence equivalent boundary conditions $UA$ and $UB$ can be written in the basis  $\{p_1,p_2\}$ as
	\begin{align*}
	A= \begin{bmatrix} A_{11} & A_{12} \\
	0 & 0
	\end{bmatrix}\quad \hbox{and} \quad 
	B= \begin{bmatrix} B_{11} & 0 \\
	B_{21} & 0
	\end{bmatrix}.
	\end{align*}
	By Lemma~\ref{lem:11rr}, one has $Q BP^{\perp}\neq 0$ and $Q^{\perp}AP\neq 0$ and hence $B_{21}\neq 0$ and $A_{12}\neq 0$. Therefore equivalent boundary conditions are 
		\begin{align*}
		A= \begin{bmatrix} A_{11} & 1 \\
		0 & 0
		\end{bmatrix}\quad \hbox{and} \quad 
		B= \begin{bmatrix} B_{11} & 0 \\
		1 & 0
		\end{bmatrix}.
		\end{align*}
	Hence
	\begin{align*}
	(A\pm ikB)= \begin{bmatrix}A \pm ik B_{11} & 1 \\
	0 & \pm ik
	\end{bmatrix}, 
	(A\pm ikB)^{-1} =  
	\frac{1}{\pm ik(A_{11} \pm ik B_{11})}  \begin{bmatrix} \pm ik  & -1 \\
	0 & A_{11} \pm ik B_{11}
	\end{bmatrix},
	\end{align*}
and
	\begin{align*}
\mathfrak{S}(k;A,B)
	= \begin{bmatrix}-\frac{A_{11}-ikB_{11}}{A_{11}+ikB_{11}}  & \frac{-2}{A_{11}+ikB_{11}}  \\
	0 & -1
	\end{bmatrix}.
	\end{align*} 
This is uniformly bounded away from the only possible pole at $k=ik/B_{11}$.
\end{proof}

Our findings are summarized in Table~\ref{fig:cayley}, where as before  $P$ is the orthogonal projection onto $\Ker B$ and $P^{\perp}=\mathds{1}-P$, and uniformly bounded refers to the Cayley transform away from its poles. 
	
\begin{table}[H]
	\begin{tabular}{c c c c c}
		$\dim \Ker A$ & $\dim \Ker B$ & Condition & Cayley transform & Ref. \\ \hline
		& & & &\\
		0 & 0 & none & uniformly bounded & Lemma~\ref{lem:msec}\\
		1 & 0 & none &  uniformly bounded & Lemma~\ref{lem:msec} \\
		2 & 0 & none & uniformly bounded & Lemma~\ref{lem:msec} \\
		0 & 2 & none &  uniformly bounded & Lemma~\ref{lem:msec} \\
		0 & 1 & $P^\perp B P^\perp\neq 0$ & uniformly bounded  & Lemma~\ref{lem:Ok}  \\
		0 & 1 & $P^\perp B P^\perp= 0$ & $\norm{\mathfrak{S}(k;A,B)}=\mathcal{O}(|k|)$  & Lemma~\ref{lem:Ok}  \\
        1 & 1 & $\Ker A\neq \Ker B$ &  uniformly bounded  & Lemma~\ref{lem:11} \\
         1 & 1 & $\Ker A= \Ker B$ &  $\mathfrak{S}(k;A,B)=\infty$   & Lemma~\ref{lem:11} \\
         & & & $A,B$ irregular & \\[1pt]
	\end{tabular}
	\caption{Cayley transforms}\label{fig:cayley}
\end{table}

\begin{comment}
{%\color{red} 
Summing up: the boundary conditions $A,B$ are regular \textit{and} the Cayley transform $\mathfrak{S}(k;A,B)$ is uniformly bounded outside a compact set if and only if
\begin{itemize}
\item $\dim \Ker B=0$, or else $\dim\Ker A=0$ and $\dim\Ker B=2$ (Lemma~\ref{lem:msec});
\item $\dim\Ker A=0$ and $\dim\Ker B=1$ with $P^\perp B P^\perp\neq 0$ (Lemma~\ref{lem:Ok});
\item $\dim\Ker A=\dim\Ker B=1$ with $\Ker A\neq \Ker B$. %and $QBP^\perp\ne 0$ 
(Lemma~\ref{lem:11}).
\end{itemize}
Here $P$ is the orthogonal projection onto $\Ker B$ and $Q^\perp$ is the orthogonal projection onto $\Ran A$.
This finding will be crucial in the following section.
}
\end{comment}

\section{Resolvent estimates}\label{sec:caybdd}

% Let $\dim \Ker A =0$ and $\dim \Ker B =1$, then  Since $A$ invertible, one can consider $A'=\mathds{1}$ and $B'=A^{-1}B$, and letting 

%For the situation considered here, this can be summarized as follows.
%Definitions: Spaces, Operators, domains, forms
%boundary conditions, regular irregular
%spectra and resolvents
The keystone of our analysis is that for regular boundary conditions the resolvent $(-\Delta(A,B)-k^2)^{-1}$ is an integral operator, i.e.,
\begin{align*}
(-\Delta(A,B)-k^2)^{-1}f(x)=\int_{[0,\infty)\times \{1,2\}} r_{A,B}(x,y;k)f(y)\ dy,\qquad 
\end{align*}
where
\begin{align*}
f=\begin{bsmallmatrix}f_1\\f_2\end{bsmallmatrix}\in L^2\big([0,\infty)\big)\oplus L^2\big([0,\infty)\big),\ x=\begin{bsmallmatrix}x_1\\ x_2\end{bsmallmatrix}, y=\begin{bsmallmatrix}y_1\\ y_2\end{bsmallmatrix}\in [0,\infty)\times \{1,2\},
\end{align*}
 with kernel 
\begin{multline}\label{eq:res_infty}
r_{A,B}(x,y;k) \\= \frac{i}{2k}\left\{\begin{bmatrix}
e^{ik|x_1-y_1|} & 0 \\ 0 & e^{ik|x_2-y_2|}
\end{bmatrix}
+
\begin{bmatrix}
e^{ikx_1} & 0 \\ 0 & e^{ikx_2}
\end{bmatrix}
\mathfrak{S}(k;A,B)
\begin{bmatrix}
e^{iky_1} & 0 \\ 0 & e^{iky_2}
\end{bmatrix}
\right\}.
\end{multline}
whenever $k\in \C$ such that $\hbox{Im } k>0$ and $A+ik B$ is invertible, cf. \cite[Prop. 4.7]{HKS2015}.  We stress that the first addend on the right hand side corresponds to the kernel of the Laplacian on $\R$ without any point interactions; the second addend can be thus interpreted as a correcting term that mirrors the influence of the point interactions. It is also remarkable that the kernel is bounded and jointly uniformly continuous on $\R\times \R$, regardless of $A,B$; in particular, it extends to a bounded linear operator from $L^1$ to $L^\infty$.

\begin{lemma}[Estimate for uniformly bounded Cayley transform]
\label{lem:resbdd}
	Let the boundary conditions $A,B$ be regular and such that $k \mapsto \mathfrak{S}(k;A,B)$ is uniformly bounded away from its poles. Then there exists  $C>0$ such that
	\begin{equation}\label{eq:analyt}
	\norm{(-\Delta(A,B)-k^2)^{-1}} \leq \frac{1}{\dist(k^2, [0,\infty))} +  \frac{C}{|k||\Im k| \dist(S,k)},
	\end{equation} 
	where $\Im k>0$ with $\det (A+ik B)\neq 0$ and 
	 $$S=\{s\in \C\colon \Im s>0 \hbox{ or } s \in [0,\infty), \hbox{ and }  s \hbox{ non-removable singularity of } \mathfrak{S}(k;A,B) \}.$$
%	 
%	
%	
%	
%	
%	and $\omega\ge 0$ such that 
%	\begin{equation}\label{eq:analyt}
%	\norm{(\Delta(A,B)+k^2)^{-1}} \leq \frac{C}{|k^2| +\omega}
%	\end{equation} 
%	for all $k^2$ in a sector of angle $>\pi/2$.
\end{lemma}
\begin{proof}
By using \eqref{eq:res_infty} we obtain outsides the poles of $\mathfrak{S}(k;A,B)$ the estimate
	\begin{align*}
	\norm{(-\Delta(A,B)-k^2)^{-1}f} \leq \frac{1}{|k|^2}\norm{f} + \frac{1}{|k|}\norm{\mathfrak{S}(k;A,B)} \cdot \norm{f} \cdot\norm{e^{ik\cdot}}^2.
	\end{align*}
The first term follows from the standard resolvent estimate for the Laplacian on $\R$ with no point interactions, while for the second one we have used the product form of the kernel, and moreover  $\norm{e^{ik\cdot}}^2=1 / (2\hbox{Im }k)$. Note that non-removable singularities of $\mathfrak{S}(k;A,B)$ are poles of order one and hence, $\norm{\mathfrak{S}(k;A,B)}\leq C \dist(S,k)$.
\end{proof}	

%{\color{red}
In the case of $m$-sectorial boundary conditions, a stronger estimate holds.
\begin{lemma}[Resolvent estimate for m-sectorial boundary conditions]\label{lem:cosine}
Let the boundary conditions $A,B$ be $m$-sectorial. Then there exist $C>0$ and $\omega \ge 0$ such that
\begin{equation}\label{eq:cosine}
\|\lambda(\Delta(A,B)-\lambda^2)^{-1}\| \le \frac{C}{\Re \lambda -\omega}
\end{equation}
all $\lambda\in \C$ with $\Re\lambda >\omega$;
in particular, the spectrum of $\Delta(A,B)$ is contained in a parabola centered on the real axis and contained in a left half-plane.
\end{lemma}

\begin{proof}
The proof is based on a result due to Lions: If a bounded, $H$-elliptic sesquilinear form $a$ with form domain $V$ satisfies the additional condition
\begin{equation}\label{eq:lions}
|\Im a(u,u)|\le M\|u\|_V \|u\|_H\qquad \hbox{for all }u\in V,
\end{equation}
then the associated operator $A$ generates a cosine operator function on $H$ with associated Kisy{\'n}ski space $V$ and its resolvent satisfies an estimate corresponding to~\eqref{eq:cosine}, see e.g.~\cite[\S~3.14]{AreBatHie01} and~\cite[\S~6.2]{Mug_book}. Hence, it  suffices to observe that~\eqref{eq:lions} is satisfied by the form $\delta_{A,B}$ defined in~\eqref{eq:form}: the proof of this fact is analogous to that of~\cite[Lemma~6.63]{Mug_book}.
\end{proof}
%}

The only cases where the Cayley transform is not uniformly bounded have been discussed in Lemma~\ref{lem:Ok}.
\begin{lemma}[Estimate for the other cases]\label{lem:resbad}
	Let $\dim \Ker A =0$ and $\dim \Ker B =1$ and let $P^{\perp}BP^{\perp}= 0$. Then $\sigma_p(-\Delta(A,B))=\emptyset$ and for some $c>0$
		\begin{align*}
		\norm{(-\Delta(A,B)+\kappa^2)^{-1}} \geq \frac{c}{\kappa^{3/2}} \quad \hbox{as} \quad \kappa\to \infty.
		\end{align*} 
		In particular $\Delta(A,B)$ is not a generator of a $C_0$-semigroup. 
\end{lemma}
\begin{proof}
 As in \cite[Section 6]{HKS2015} one can show that $\Delta(A,B)$ is unitarily equivalent to  $\Delta(A',B')$ with
	\begin{align*}
	\mathfrak{S}(k;A',B')
	% &=-\frac{1}{1 + ik PBP }  \begin{bmatrix} 1 & 0 \\
	% -ik P^{\perp} BP & 1 + ik PBP
	% \end{bmatrix}
	% \begin{bmatrix}1 - ik PBP & 0 \\
	% -ik P^{\perp} BP & 1
	% \end{bmatrix}\\
	=  \begin{bmatrix}-1   & 0 \\
	2ik B_{21} & 1
	\end{bmatrix}, \quad B_{21}\neq 0,
	\end{align*}
	and there are no eigenvalues nor poles of $\mathfrak{S}(k;A',B')$.
	Consider the function $u=(u_1,u_2)^T=(\chi_{[0,1]},0)^T$, where $\chi_{[0,1]}$ denotes the characteristic function of the unit interval. Then 
	 \begin{multline}\label{eq:res_infty-2}
	 (-\Delta(A,B)-k^2)^{-1}u = 
	  \frac{i}{2k}\left\{\begin{bmatrix}
	 \int_0^{1} e^{ik|x_1-y_1|} dy_1 \\ 0 \end{bmatrix} 
	 +
	 \begin{bmatrix}
	 e^{ikx_1} \\ 2ik B_{12}   e^{ikx_2}
	 \end{bmatrix}\int_0^1 e^{iky_1} dy_1
	 \right\}, 
	 \end{multline}
	 and estimating the second component only
	 \begin{align*}
	 \norm{(-\Delta(A,B)-k^2)^{-1}u} \geq  
	   |B_{12}|\cdot \abs{\int_0^1 e^{iky_1} dy_1}\cdot  \norm{ e^{ikx_2}} = |B_{12}| \frac{|e^{ik}-1|}{|k|} \frac{1}{(2\Im k)^{1/2}}. 
	 \end{align*}
	 In particular for $k= i\kappa$, $\kappa>0$,
	 \begin{align*}
 \frac{|e^{-\kappa}-1|}{\sqrt{2}|\kappa|^{3/2}}\to \frac{1}{\sqrt{2}|\kappa|^{3/2}}   \quad \hbox{as} \quad \kappa \to \infty, 
	 \end{align*}
	  and therefore
	   \begin{align*}
	   \norm{(\Delta(A,B)-\kappa^2)^{-1}} \geq \mathcal{O}(|\kappa|^{-3/2}).
	   \end{align*}
	   In particular, assume that $\Delta(A,B)$ is the generator of a $C_0$-semigroup, then 
	   \begin{align*}
	    \mathcal{O}(|\kappa|^{-3/2}) \leq \norm{\Delta(A,B)-\kappa^2)^{-1}} \leq M/(\kappa^2-\omega), \quad \hbox{for } \omega >0 \hbox{ and } M>0,
	   \end{align*}
	   which multiplying by $\kappa^2$ and passing to the limit $\kappa\to \infty$ leads to a contradiction. Recall that $\Delta(A,B)$ is closed and densely defined. 
\end{proof}

\section{Invariance properties}\label{sec:invariance}

Several issues in the qualitative analysis of semigroups associated with sesquilinear forms are made particularly easy by variational methods. In particular,  the classical Beurling--Deny criteria have been generalized in~\cite{Ouh96}; based upon this general criterium, invariance properties for heat equations on metric graphs have been obtained in~\cite{CarMug09}. We can paraphrase~\cite[Prop.~4.3]{CarMug09} (see also~\cite[Thms.~6.71 and 6.72]{Mug_book} and obtain the following: given a closed convex subset $C$ of $\C^2$, we denote by $\mathcal C$ the induced closed convex subset of $L^2\big([0,\infty)\big)\oplus L^2\big([0,\infty)\big)$ defined by
\[
\mathcal C:=\{f\in L^2\big([0,\infty)\big)\oplus L^2\big([0,\infty)\big):f(x)\in C\hbox{ for a.e.}x\in [0,\infty)\dot{\cup} [0,\infty) \}
\]
 
\begin{proposition}\label{prop:invar-m-sect}
Let the boundary conditions $A,B$ be m-sectorial. Let $C$ be a closed convex subset of $\C^2$ with $(0,0)\in C$. 

Then the semigroup generated by $\Delta(A,B)$ leaves $\mathcal C$ invariant if and only if both the projection onto $\Ran B$ and the semigroup generated by $A+B-\mathds{1}$  leave $C$ invariant.
\end{proposition}

The power of our approach lies in the possibility of the explicit representation~\eqref{eq:res_infty} of the resolvent kernel. From this, some semigroup properties can be derived even for boundary conditions that are not m-sectorial,  when Ouhabaz' variational methods are not available.

\begin{lemma}\label{lem:invar-b}
Let the boundary conditions be regular and $\Delta(A,B)$ generate a contractive $C_0$-semigroup. Let $C$ be a closed convex subset of $\C^2$. Then the semigroup generated by $\Delta(A,B)$ leaves $\mathcal C$ invariant provided %$-\kappa^2 r_{A,B}(x,y;i\kappa)$
\begin{equation*}
\begin{split}
%&\kappa^2 r_{A,B}(x,y;i\kappa) \\
&%\quad=
\frac{\kappa}{2}
\begin{bmatrix}
e^{-\kappa|x_1-y_1|} + e^{-\kappa (x_1+y_1)}\sigma_{11}(i\kappa) & e^{-\kappa (x_1+y_2)} \sigma_{12}(i\kappa)\\ 
e^{-\kappa (x_2+y_1)} \sigma_{21}(i\kappa) & e^{-\kappa|x_2-y_2|} + e^{-\kappa (x_1+y_2)} \sigma_{22}(i\kappa)
\end{bmatrix}
\end{split}
\end{equation*}
leaves $C$ invariant for all $\kappa>0$ and all $x_i,y_j\in [0,\infty)$, where $\sigma_{ij}(i\kappa)=(\mathfrak{S}(i\kappa;A,B))_{ij}, 1\leq i,j\leq 2$.
\begin{comment}
\begin{equation}\label{eq:kernel-real}
\frac{\sqrt{\lambda}}{2}
\left\{\begin{bmatrix}
e^{-\sqrt{\lambda}|x_1-y_1|} & 0 \\ 0 & e^{-\sqrt{\lambda}|x_2-y_2|}
\end{bmatrix}+
\begin{bmatrix}
e^{-\sqrt{\lambda}x_1} & 0 \\ 0 & e^{-\sqrt{\lambda}x_1}
\end{bmatrix}
\mathfrak{S}(i\sqrt{\lambda};A,B)
\begin{bmatrix}
e^{-\sqrt{\lambda}y_1} & 0 \\ 0 & e^{-\sqrt{\lambda}y_1}
\end{bmatrix}\right\}
\end{equation}
leaves $C$ invariant for all $\lambda>0$.
\end{comment}
In particular, the semigroup generated by $\Delta(A,B)$ is $L^\infty$-contractive provided $\mathds{1}+\mathfrak{S}(i\kappa;A,B)$ leaves $\{\xi \in \C^2:|\xi_1|+| \xi_2|\le 1\}$ invariant for all $\kappa>0$.
\end{lemma}

\begin{proof}
It is well-known that the semigroup leaves $\mathcal C$ invariant if and only if so does $\lambda (\lambda-\Delta(A,B))^{-1}$ for all $\lambda>0$, see~\cite[Prop.~2.3]{Ouh96}. By~\eqref{eq:res_infty}, $
\kappa^2 r_{A,B}(x,y;i\kappa)$ is the kernel of $\lambda (\lambda-\Delta(A,B))^{-1}$ for $\lambda=-(i\kappa)^2$; a direct computation shows that  for all $\mu\in \mathbb R$
\begin{equation}\label{eq:res_infty-bis}
\begin{split}
\mu^2 r_{A,B}(x,y;i\kappa)=
\frac{\mu^2}{2\kappa}
\begin{bmatrix}
e^{-\kappa|x_1-y_1|} + e^{-\kappa (x_1+y_1)}\sigma_{11}(i\kappa) & e^{-\kappa (x_1+y_2)} \sigma_{12}(i\kappa)\\ 
e^{-\kappa (x_2+y_1)} \sigma_{21}(i\kappa) & e^{-\kappa|x_2-y_2|} + e^{-\kappa (x_1+y_2)} \sigma_{22}(i\kappa)
\end{bmatrix}
\end{split}
\end{equation}
and the main claim now follows from closedness of $\mathcal C$, taking $\mu=\kappa$.

Finally, in order to prove the assertion about $L^\infty$-contractivity,  it suffices to observe that the matrix in~\eqref{eq:res_infty-bis} is in absolute value no larger than $\mathds{1}+\mathfrak{S}(i\kappa;A,B)$.
\end{proof}

If $\mathcal C$ is the positive cone, then the assertion can be sharpened as follows.

\begin{corollary}\label{cor:posit-equiv}
Let the boundary conditions be regular and $\Delta(A,B)$ generate a \emph{quasi}-contractive $C_0$-semigroup. Then the semigroup generated by $\Delta(A,B)$ is real if and only if $\mathfrak{S}(i\kappa;A,B)$ has real entries; in this case, the semigroup is additionally positive  whenever $\mathds{1}+\mathfrak{S}(i\kappa;A,B)$ leaves $\{\xi \in \R^2:\xi_1\ge 0 , \ \xi_2\ge 0\}$ invariant for some $\kappa_0$ and all $\kappa\ge \kappa_0$.
\end{corollary}
\begin{proof}
Positivity and reality of a semigroup is unaffected by scalar (real) perturbations of its generator. Furthermore, reality (resp., positivity) of a positive operator is equivalent to reality (resp., positivity) of its kernel, cf.~\cite[Thm.~5.2]{MugNit11}. Finally,~\eqref{eq:res_infty-bis} shows that the entries of the resolvent's kernel at $i\kappa$ are real if and only if so are the entries of $\mathfrak{S}(i\kappa;A,B)$. The proof of the positivity follows essentially the proof of \cite[Thm.~4.6]{KS2006} although there only self-adjoint boundary conditions are considered:  we only need to observe that if $\mathds{1}+\mathfrak{S}(i\kappa;A,B)$ has real and positive entries, then for $\mu=\kappa$ the matrix in~\eqref{eq:res_infty-bis} is entry-wise no smaller than
\begin{equation}\label{eq:res_infty-ter}
\frac{\kappa}{2}
\begin{bmatrix}
e^{-\kappa (x_1+y_1)}\big(1+\sigma_{11}(i\kappa)\big) & e^{-\kappa (x_1+y_2)} \sigma_{12}(i\kappa)\\ 
e^{-\kappa (x_2+y_1)} \sigma_{21}(i\kappa) & e^{-\kappa (x_2+y_2)} \big(1+\sigma_{22}(i\kappa)\big)
\end{bmatrix},
\end{equation}
whence the claim follows.
\end{proof}

\begin{example}\label{ex:sec-revis}
By Proposition~\ref{prop:invar-m-sect}, the boundary conditions in Example~\ref{ex:sec} define a semigroup that leaves invariant $\mathcal C$ if and only if the semigroup generated by $A$ leaves $C$ invariant: e.g.,\ the semigroup generated by $\Delta(A,B)$ is positive if and only if $A_{11}$ is real and $A_{21}\ge 0$. That this is a sufficient condition can be deduced from Corollary~\ref{cor:posit-equiv}, too.

Furthermore, the semigroup is $L^\infty$-contractive (and in this case automatically $L^p$-contractive for all $p\in [1,\infty]$) if and only if $\Re A_{11}\le 0$ and  $A_{12}=0$, cf.~\cite[Lemma~6.1]{Mug07}. (Observe that the latter condition induces a decoupling of our system, as we are left with two Laplacians on $[0,\infty)$ with Neumann and Robin boundary conditions, respectively.) 
\end{example}

\begin{remark}\label{rem:hks}
	One criterion for contractivity of the semigroup is that $A,B$ are m-sectorial with $\Re L \leq 0$, see \cite[Thm. 2.4]{KSP2008}, or equivalently $\Re AB^* \leq 0$, but as we know the case of m-sectorial boundary conditions can be treated more directly by Proposition~\ref{prop:invar-m-sect}, without invoking Lemma~\ref{lem:invar-b}. Example~\ref{ex:new} shows that non-(quasi-)contractive semigroups can actually arise; we note in passing that in the setting of Example~\ref{ex:new}
	\begin{equation}\label{eq:hks}
\mathfrak{S}(i\kappa;A_\tau,B_\tau)=\frac{1}{\cos \tau}\begin{bmatrix}
i\sin\tau & 1\\1 & -i\sin\tau
\end{bmatrix},\qquad \kappa>0,
	\end{equation}
	(see~\cite[Example~3.5]{HKS2015}), i.e., $\mathds{1}+\mathfrak{S}(i\kappa;A,B)$ does not leave either the positive cone of $\R^2$ or the unit ball of $\ell^\infty\times \ell^\infty$ invariant.
\end{remark}

An interesting consequence of Theorem~\ref{thm:generator}.(a) is that if there is  a simple pole $k_0$ of $k\mapsto (A+ikB)^{-1}$ with $\Im k_0>|\Re k_0|$, then the peripheral spectrum of $\Delta(A,B)$ is finite and consists of simple poles of the resolvent. This paves the way to study semigroups that are merely \textit{asymptotically} positive; i.e., those semigroups whose orbits starting at positive initial data tend to the lattice's positive cone, see~\cite[Def.~8.1]{DGK16_b}.

%{\color{red}
\begin{proposition}\label{prop:asympt}
{%\color{red}
Let $\Delta(A,B)$ generate a $C_0$-semigroup. Assume the zero $k$ of $\{k:\Im k>0\}\ni k\mapsto \det (A+ikB)\in \C$ of larger magnitude lies on $i(0,\infty)$, i.e., $k=i\kappa_0$ for some $\kappa_0>0$, and let $A\ne \kappa_0 B$. Consider the following assertions:}
%\footnote{\DM{Dadurch gibt es einen reellen, strikt negativen Eigenwert. Einfachheit von $i\kappa$ als Pol von $\mathfrak{S}$ ist zur Einfachheit von $\kappa^2$ als Eigenwert von $\Delta(A,B)$ doch äquivalent, oder?} \AH{Ja, siehe Bemerkung oben Rem \ref{rem:mult}} \DM{Danke, aber gerade dank Rem. \ref{rem:mult} sehe ich jetzt, dass die Bedingung (``$i\kappa_0$ ist einfacher Pol'') zu stark war; denn im Beispiel~\ref{exa:hyperb} ist $i$ ein doppelter Pol von $\mathfrak{S}$ und dennoch ist $1=-i^2$ ein einfacher Eigenwert, da $A-B\ne 0$. Umso besser!}  }
%; let us denote it by $i\kappa_0$, $\kappa_0>0$. {\color{red}Assume furthermore that $A\ne \kappa_0 B$ and} consider the following assertions:
\begin{enumerate}[(i)]
\item the semigroup generated by $\Delta(A,B)$ is asymptotically positive;
 \item the spectral projection of $\Delta(A,B)$ associated with $\kappa_0^2$ is positive;
  \item the distance to the set $[0,\infty)$ of each entry of $(\kappa-\kappa_0)^2\mathfrak{S}(i\kappa;A,B)$ tends to $0$ as $\kappa\searrow \kappa_0$;
  \item $\lim\limits_{\kappa\searrow \kappa_0}\frac{(\kappa-\kappa_0)^2}{\det(A-\kappa B)}=0$.
\end{enumerate}
Then $(i)\Leftrightarrow (ii)\Leftarrow (iii) \Leftarrow (iv)$. 
%\footnote{\DM{Amru, ich glaube, dass $(iv)$ automatisch erfüllt ist, sobald $\kappa_0$ ein einfacher Pol von $\mathfrak{S}$ ist. Das würde allerdings bedeuten (\textit{wenn alles stimmt!}), dass ALLE $A,B$, die die Annahmen dieses Satzes erfüllen, asymptotisch positive Halbgruppen erzeugen! Dies scheint mir sehr überraschend zu sein. Die einzige erklärung, die mir einfällt, ist dass unsere Halbgruppen als kleine Störung der Gaußschen Halbgruppe auf $\R$ verstanden werden können; für große Zeiten verschwinden dann die Effekte der Randbedingung und somit konvergieren die Orbits gegen den positiven Kegel. Dennoch wäre ich Dir sehr dankbar, wenn Du alles überprüfen könntest. Wenn dies allerdings stimmt, dann können wir natürlich auf ein Beispiel verzichten.}\AH{Ja falls nicht Robin-RB, siehe Rem~\ref{rem:mult}! }}
\end{proposition}
We refer to~\cite[\S~IV.1]{EngelNagel} for the definition of spectral projections of possibly non-selfadjoint operators.
%\footnote{\AH{Braucht man hier als Voraussetzung noch das die halbgruppe reell ist??? Braucht man einen reellen eigenwert bzw. eine positivie eigenfunktion?} \DM{Ich habe noch mal nachgeprüft und nein, es sieht aus, dass man nicht voraussetzt, dass die HG reell ist. Man erlaubt auch, dass eine komplexwertige Bahn gegen den positiven Kegel strebt. Positive Eigenfunktion? Auch nicht! Das ist gerade der Vorteil gegenüber der schließlichen Positivität, die schwieriger zu überprüfen ist; denn Positivität der Eigenfunktion folgt eben aus der asymptotischen Positivität der Resolvente, s.~\cite[Thm.~7.6]{DGK16_b}. Reeller Eigenwert? Ja, sorry, blöder Fehler von mir. Das ist natürlich nötig. Ich habe korrigiert. }}
%Clearly, condition (ii) is especially satisfied if each (possibly complex) entry of the matrix $(\kappa-\kappa_0){\mathfrak{S}}(i\kappa;A,B)=(\kappa_0-\kappa)(A-\kappa B)^{-1}(A+\kappa B)$ 
%\begin{itemize}
%\item is positive for some $\varepsilon>0$ and all $\kappa\in ]\kappa_0,\kappa_0+\epsilon[$, or
%\item tends to 0 as $\kappa\searrow \kappa_0$.
%\end{itemize}
\begin{proof}
By Theorem~\ref{thm:generator} and Remark~\ref{rem:mult}, the main assumptions imply that the peripheral spectrum contains precisely one eigenvalue, which is simple: hence the spectral bound is a dominant spectral value. We are thus in the position to apply~\cite[Thm.~1.2]{DGK16_b}: in view of~\eqref{eq:res_infty-bis} we conclude that asymptotic positivity (in the sense of~\cite[Def.~7.2]{DGK16_b}) of  $\lambda\mapsto (\lambda-\Delta(A,B))^{-1}$ at $\lambda_0:=-(i\kappa_0)^2=\kappa_0^2$ (which is in turn equivalent to (ii), by~\cite[Thm.~7.6]{DGK16_b}) is equivalent to the condition that the distance to $[0,\infty)$ of each entry of 
\begin{equation}\label{eq:fin}
\begin{split}
\frac{(\kappa-\kappa_0)^2}{2\kappa}
\begin{bmatrix}
e^{-\kappa|x_1-y_1|} + e^{-\kappa (x_1+y_1)}\sigma_{11}(i\kappa) & e^{-\kappa (x_1+y_2)} \sigma_{12}(i\kappa)\\ 
e^{-\kappa (x_2+y_1)} \sigma_{21}(i\kappa) & e^{-\kappa|x_2-y_2|} + e^{-\kappa (x_1+y_2)} \sigma_{22}(i\kappa)
\end{bmatrix}
%
%\begin{bmatrix}
%e^{-\kappa (x_1+y_1)}\big(1+\sigma_{11}(i\kappa)\big) & e^{-\kappa (x_1+y_2)} \sigma_{12}(i\kappa)\\ 
%e^{-\kappa (x_2+y_1)} \sigma_{21}(i\kappa) & e^{-\kappa (x_2+y_2)} \big(1+\sigma_{22}(i\kappa)\big)
%\end{bmatrix},
\end{split}
\end{equation}
 tends to 0 as $\kappa\searrow \kappa_0$ for each $x_1,x_2,y_1,y_2\in [0,\infty)$. 
Now, observe that if $K$ is a cone in a lattice $X$, $\delta\in (0,1]$, and $a\in K$, then for any $b\in X$ $\dist(b,K)\ge \dist(a+\delta b,K)$; we conclude that the distance to the positive cone $\C^2_+$ of the matrix in~\eqref{eq:fin} is no larger than the distance to the same cone of the matrix $(\kappa-\kappa_0)^2\mathfrak{S}(i\kappa;A,B)$, which proves that $(i)$ is implied by $(iii)$. To conclude the proof, it suffices to observe that the poles of $k\mapsto \mathfrak{S}(i\kappa;A,B)$ are the zeros of $k\mapsto \det(A+i\kappa B)$, with equal multiplicity.
%
%Finally, let $M(\kappa):=\det(A-\kappa B)\mathfrak{S}(i\kappa;A,B)$: since $\mathfrak{S}$ only has poles in the zeros of $\det(A-\kappa B)$, $\kappa\mapsto M(\kappa)$ is bounded and $(iv)$ implies $(iii)$.
\end{proof}

%{\color{red}
\begin{example}\label{exa:asympt}
Let us come back to the setting in Example~\ref{ex:sec}. We have already seen that if $A_{11}>0$, then $A_{11}^2$ is a dominant eigenvalue of $\Delta(A,B)$ and we can hence apply Proposition~\ref{prop:asympt} to~\eqref{eq:cayley_msec-exa} and conclude that the semigroup generated by $\Delta(A,B)$ is asymptotically positive for any $A_{11},A_{12}$.
%\begin{itemize}
%\item $\frac{(A_{11}-\kappa)^2}{\kappa}\frac{A_{11} + \kappa}{A_{11}-\kappa}=A_{11}^2-\kappa^2$ and
%\item $(A_{11}-\kappa)^2\frac{A_{12}}{\kappa^2} (\frac{A_{11} + \kappa}{A_{11}-\kappa}  -1)=2A_{12} $
%\end{itemize}
%tends to $0$ as $\kappa\searrow A_{11}$.
%Thus we are able to prove asymptotic positivity only in those cases ($A_{11}\in\R, A_{12}\ge 0$) were in fact plain positivity is already known. \footnote{\DM{Wie gesagt: eigentlich habe ich gehofft, asymptotische Positivität nachweisen zu können in Fällen, in der Positivität nicht gegeben ist. So finde ich dieses Beispiel leider nur mäßig interessnt. Was ist Deine Meinung? Sollen wir es streichen? Wenn meine Zweifel über den Beweis von Proposition~\ref{prop:asympt} berechtigt sind, können wir nicht mal sagen, dass asymptotische Positivität bereits Positivität impliziert. Allerdings würde es anders aussehen, wenn wir den Spektralprojektor für den dominanten Eigenwert bestimmen könnten.}}
\end{example}

{%\color{red}
\begin{example}\label{exa:hyperb}
%A graph consisting of two edges $\me_1,\me_2\simeq [0,\ell]$ with boundary conditions
%\[
%\psi_1(0)=0,\quad \psi_2'(0)=0,\quad \psi_1'(\ell)=\psi_2(\ell),\quad \psi_2'(\ell)=\psi_1(\ell).
%\]
%was studied in~\cite[Exa.~3.12]{BerLatSuk18}. The key feature of these so-called ``hyperbolic'' boundary conditions is that $1$ is an eigenvalue of the Laplacian regardless of $\ell\in (0,\infty)$. 

%Let us treat the transition conditions at $\ell$ in our setting: in our formalism this corresponds to
Consider
\[
A=\begin{bmatrix}
0 & -1\\ -1 & 0
\end{bmatrix}\quad\hbox{and }\quad B=\begin{bmatrix}
1 & 0 \\ 0 & 1
\end{bmatrix}.
\]
	%	\AH{Irgendwas stimmt hier nicht ganz: $\det(A+ik B)=-k^2-1$ also ist $k=i$ eine Nullstelle und die einzige mit $\Im k>0$ also ist $-1$ ein eigenwert und das $\Ker A -B= \hbox{span} (1,1)^T$ ist die eigenfunktion $\psi=(e^{-x_1}, e^{-x_2})$, der eigenprojektor ist $(1/\norm{\psi}^2)\psi \otimes \psi$ also positiv... Wie passt das zu deiner rechnung?}	
Because $\det(A+ikB)= -k^2-1$, we obtain that $1=-(i\kappa_0)^2$ with $\kappa_0=1$ is the only eigenvalue of $\Delta(A,B)$; and it is simple, since $A\ne B$.
Accordingly, by Proposition~\ref{prop:invar-m-sect} the semigroup generated by $\Delta(A,B)$ is not positive (nor it is $\ell^\infty$-contractive). Let us sharpen this assertion: 
\[
\frac{(\kappa-\kappa_0)^2}{\det(A-\kappa B)}=\frac{(\kappa-1)^2}{\kappa^2-1}=\frac{\kappa-1}{\kappa+1}\stackrel{\kappa\searrow 1}{\longrightarrow}0,
\]
and we conclude from Proposition~\ref{prop:asympt} that the semigroup generated by $\Delta(A,B)$ is asymptotically positive.
\end{example}
}

\bibliographystyle{alpha}
\bibliography{literatur_SOTA2018}

\begin{thebibliography}{KKVW09}

\bibitem[ABHN01]{AreBatHie01}
W.\ Arendt, C.J.K.\ Batty, M.\ Hieber, and F.\ Neubrander.
\newblock {\em Vector-{V}alued {L}aplace {T}ransforms and {C}auchy {P}roblems},
  volume~96 of {\em Monographs in Mathematics}.
\newblock Birkh{\"a}user, Basel, 2001.

\bibitem[Bir08a]{Birkhoff1908_a}
George~D. Birkhoff.
\newblock Boundary value and expansion problems of ordinary linear differential
  equations.
\newblock {\em Trans. Amer. Math. Soc.}, 9(4):373--395, 1908.

\bibitem[Bir08b]{Birkhoff1908_b}
George~D. Birkhoff.
\newblock On the asymptotic character of the solutions of certain linear
  differential equations containing a parameter.
\newblock {\em Trans. Amer. Math. Soc.}, 9(2):219--231, 1908.

\bibitem[BK13]{BeKu_book}
Gregory Berkolaiko and Peter Kuchment.
\newblock {\em {Introduction to Quantum Graphs}}, volume 186 of {\em
  Mathematical Surveys and Monographs}.
\newblock American Mathematical Society, Providence, RI, 2013.

\bibitem[CM09]{CarMug09}
S.\ Cardanobile and D.\ Mugnolo.
\newblock Parabolic systems with coupled boundary conditions.
\newblock {\em J.\ Differ.\ Equ.}, 247:1229--1248, 2009.

\bibitem[DGK16]{DGK16_b}
Daniel Daners, Jochen Gl\"{u}ck, and James~B. Kennedy.
\newblock Eventually and asymptotically positive semigroups on {B}anach
  lattices.
\newblock {\em J. Differential Equations}, 261(5):2607--2649, 2016.

\bibitem[DS71]{DSIII}
Nelson Dunford and Jacob~T. Schwartz.
\newblock {\em Linear operators. {P}art {III}: {S}pectral operators}.
\newblock Interscience Publishers [John Wiley \& Sons, Inc.], New
  York-London-Sydney, 1971.
\newblock With the assistance of William G. Bade and Robert G. Bartle, Pure and
  Applied Mathematics, Vol. VII.

\bibitem[EE87]{EE_book}
D.~E. Edmunds and W.~D. Evans.
\newblock {\em Spectral theory and differential operators}.
\newblock Oxford Mathematical Monographs. The Clarendon Press, Oxford
  University Press, New York, 1987.
\newblock Oxford Science Publications.

\bibitem[EN00]{EngelNagel}
Klaus-Jochen Engel and Rainer Nagel.
\newblock {\em One-parameter semigroups for linear evolution equations}, volume
  194 of {\em Graduate Texts in Mathematics}.
\newblock Springer-Verlag, New York, 2000.
\newblock With contributions by S. Brendle, M. Campiti, T. Hahn, G. Metafune,
  G. Nickel, D. Pallara, C. Perazzoli, A. Rhandi, S. Romanelli and R.
  Schnaubelt.

\bibitem[Gom99]{Gom99}
A.M. Gomilko.
\newblock Conditions on the generator of a uniformly bounded {$C_0$}-semigroup.
\newblock {\em Funct.\ Anal.\ Appl.}, 33:294--296, 1999.

\bibitem[HKS15]{HKS2015}
Amru Hussein, David Krej\v{c}i\v{r}\'{\i}k, and Petr Siegl.
\newblock Non-self-adjoint graphs.
\newblock {\em Trans. Amer. Math. Soc.}, 367(4):2921--2957, 2015.

\bibitem[Hus14]{Hussein2014}
Amru Hussein.
\newblock Maximal quasi-accretive {L}aplacians on finite metric graphs.
\newblock {\em J. Evol. Equ.}, 14(2):477--497, 2014.

\bibitem[KKVW09]{KanKlaVoi09}
U.\ Kant, T.\ Klau{\ss}, J.\ Voigt, and M.\ Weber.
\newblock Dirichlet forms for singular one-dimensional operators and on graphs.
\newblock {\em J.\ Evol.\ Equ.}, 9:637--659, 2009.

\bibitem[KPS08]{KSP2008}
Vadim Kostrykin, J\"{u}rgen Potthoff, and Robert Schrader.
\newblock Contraction semigroups on metric graphs.
\newblock In {\em Analysis on graphs and its applications}, volume~77 of {\em
  Proc. Sympos. Pure Math.}, pages 423--458. Amer. Math. Soc., Providence, RI,
  2008.

\bibitem[KS06]{KS2006}
Vadim Kostrykin and Robert Schrader.
\newblock Laplacians on metric graphs: eigenvalues, resolvents and semigroups.
\newblock In {\em Quantum graphs and their applications}, volume 415 of {\em
  Contemp. Math.}, pages 201--225. Amer. Math. Soc., Providence, RI, 2006.

\bibitem[MN11]{MugNit11}
D.\ Mugnolo and R.\ Nittka.
\newblock Properties of representations of operators acting between spaces of
  vector-valued functions.
\newblock {\em Positivity}, 15:135--154, 2011.

\bibitem[Mug07]{Mug07}
D.\ Mugnolo.
\newblock Gaussian estimates for a heat equation on a network.
\newblock {\em Networks Het.\ Media}, 2:55--79, 2007.

\bibitem[Mug10]{Mug10}
D.\ Mugnolo.
\newblock Vector-valued heat equations and networks with coupled dynamic
  boundary conditions.
\newblock {\em Adv.\ Diff.\ Equ.}, 15:1125--1160, 2010.

\bibitem[Mug14]{Mug_book}
Delio Mugnolo.
\newblock {\em Semigroup methods for evolution equations on networks}.
\newblock Understanding Complex Systems. Springer, Cham, 2014.

\bibitem[Ouh96]{Ouh96}
E.M. Ouhabaz.
\newblock Invariance of closed convex sets and domination criteria for
  semigroups.
\newblock {\em Potential Analysis}, 5:611--625, 1996.

\bibitem[Ouh05]{Ouh05}
E.M. Ouhabaz.
\newblock {\em Analysis of {H}eat {E}quations on {D}omains}, volume~30 of {\em
  Lond.\ Math.\ Soc.\ Monograph Series}.
\newblock Princeton Univ.\ Press, Princeton, NJ, 2005.

\end{thebibliography}

% Moreover, for some cases where $\Delta(A,B)$ generates a $C_0$-semigroup the questions is addressed if the semigroup is eventually positive, and a criterion depending on the Cayley transform of the boundary conditions is derived. While the classical notion of positivity requires that an operator semigroup $(e^{tA})_{t\geq0}$  maps the positive cone into itself for any time $t\geq 0$, $(e^{tA})_{t\geq0}$ is called \textit{eventually positive} if, roughly speaking, any orbit $t\mapsto e^{tA}u_0$ associated with a nonnegative initial condition $u_0$ reaches the positive cone, \textit{possibly only after a transient}, i.e., for some $t_0\ge 0$ and all $t\ge t_0$. This notion has been introduced and investigated in \cite{DGK16_b}, and one illustrative example for this new notion is the one-dimensional Laplacian with non-local Robin-type boundary conditions, which is in a sense generalized here.   We stress that the results in~\cite{DGK16_b}, especially~\cite[Thm.~6.7]{DGK16_b}, cannot be applied in our context, since it is known that their assumptions imply compactness of the semigroup (see~\cite[Thm.~10.2.1]{Glu16}) which does not hold here. We will however see how to deduce eventual positivity from an explicit resolvent formula.

\end{document}